\documentclass[a4paper]{article}

\usepackage[mathscr]{eucal}
\usepackage{srcltx}
\usepackage{epsfig}
\usepackage{amssymb}
\usepackage{longtable}
\usepackage{amssymb}
\usepackage{amscd,amsfonts}
\usepackage{amsthm}
\usepackage{amsmath}
\usepackage[all]{xy}
\usepackage{mathrsfs}

\usepackage{latexsym}

\theoremstyle{plain}
\newtheorem{theorem}{Theorem}

\newtheorem{corollary}[theorem]{Corollary}

\newtheorem{lemma}[theorem]{Lemma}

\newtheorem{proposition}[theorem]{Proposition}

\theoremstyle{definition}
\newtheorem{definition}[theorem]{Definition}
\newtheorem{remark}[theorem]{Remark}

\numberwithin{theorem}{section}

\newcommand{\Hom}{\mbox{Hom}}
\newcommand{\ihom}{\mathscr{H}om}
\newcommand{\iHom}{\mathscr{H}om_{\textrm{qc}}}
\newcommand{\res}{\operatorname{res}}
\newcommand{\Ker}{\mbox{Ker}}

\newcommand{\Ext}{\mbox{Ext}}
\newcommand{\Filt}{\mbox{Filt}}

\newcommand{\Pure}{\mathbf{Pure}}
\newcommand{\Flat}{\mathcal Flat}
\newcommand{\Purefp}{\mathbf{Pure_{fp}}}
\newcommand{\Flatfp}{\mathcal Flat_{fp}}
\newcommand{\Pinjfp}{\mathcal Pinj_{fp}}
\newcommand{\Pinj}{\mathcal Pinj}
\newcommand{\Flatt}{\mathcal Flat_{\otimes}}

\def\O{\mathcal{O}}
\def\Qco{\mathfrak{Qcoh}}
\def\Mod{\mbox{-Mod}}

\def\scrM{\mathscr{M}}
\def\scrN{\mathscr{N}}
\def\scrA{\mathscr{A}}
\def\scrF{\mathscr{F}}
\def\scrT{\mathscr{T}}
\def\scrH{\mathscr{H}}
\def\scrG{\mathscr{G}}
\def\scrST{\mathscr{S}}
\def\scrR{\mathscr{R}}
\def\scrE{\mathscr{E}}

\begin{document}
\title{Pure injective and absolutely pure sheaves }
\author{ Edgar Enochs \\ enochs@ms.uky.edu \\
Department of Mathematics \\
University of Kentucky \\ 40506 Lexington, Kentucky \\ USA \\ \\
Sergio Estrada \\ sestrada@um.es \\
Departamento de Matem\'aticas \\
Universidad de Murcia \\ 30100 Murcia \\ SPAIN \\ \\
 Sinem Odaba\c{s}{\i} \\  \ sinem.odabasi1@um.es \\
\ \
Departamento de Matem\'aticas \\
\ \ \ Universidad de Murcia \\
\ \ 30100 Murcia \\
SPAIN \footnote{Estrada and Odaba\c{s}{\i}\  have been supported by the research grant 18394/JLI/13 by the Fundaci\'on S\'eneca-Agencia de Ciencia y Tecnolog\'{\i}a de la Regi\'on de Murcia in the framework of III PCTRM 2011-2014. Odaba\c{s}{\i} has been supported by the  Consejer\'{\i}a de Industria, Empresa e Innovaci\'{o}n de la CARM by means of Fundaci\'{o}n S\'{e}neca, the program of becas-contrato predoctorales de formaci\'{o}n del personal investigador, 15440/FPI/10.} }
\date{}
\maketitle
\thispagestyle{empty}
\begin{abstract}
We study two notions of purity in categories of sheaves: the categorical and the geometric. It is shown that pure injective envelopes exist in both cases under very general assumptions on the scheme. Finally, we introduce the class of locally absolutely pure (quasi--coherent) sheaves, with respect to the geometrical purity, and characterize locally Noetherian closed subschemes of a projective scheme in terms of the new class.
\end{abstract}

{\footnotesize{\it 2010 Mathematics Subject Classification. 18E15,18C35,18D10,18F20,14A15,16D90.}
}%

{\footnotesize{\it Key words and phrases}: Local purity, categorical purity, concentrated scheme, locally finitely presented category, pure injective sheaf, absolutely pure sheaf.}

\section{Introduction}
The
history of purity goes back to the work of \cite{Prufer} for abelian
groups. Later the notion was introduced into  module categories by
\cite{Cohn}. The notion, which is a problem of solving equations
with one variable in abelian groups turned out to be that of solving equations in several
variables in module categories. The notion was developed further in
\cite{Fieldhouse}, \cite{S1}, \cite{S}, \cite{Sten}. More recently it was shown by Crawley-Boevey in \cite{CB} that locally finitely presented additive categories were the most general additive setup in which to define a good purity theory. We recall that a short exact sequence in a locally finitely presented category is said to be pure whenever it is projectively generated by the class of finitely presented objects.  Several
problems in algebra and in relative homological algebra
can be solved by purity arguments. For instance to show that a class of objects of a class $\mathcal F$ in a locally finitely presented category $\mathcal A$ allows to define unique up to homotopy minimal resolutions, it suffices to check that $\mathcal F$ is closed under direct limits and under pure subobjects or under pure quotients (see \cite{Bashir,CPT}). In the case of the category of $R$-modules ($R$ any commutative ring with identity) it is well-known that purity can be also defined in terms of the tensor product, that is, a short exact sequence $\mathbb{E}=0\to L\to M\to N\to 0$ of $R$-modules is pure provided that the functor $\mathbb{E}\otimes T$ leaves the sequence exact, for each $R$-module $T$. But in general for an arbitrary monoidal locally finitely presented category these two notions need not be equivalent. For instance, when $X$ is a concentrated (i. e. quasi--compact and quasi--separated) scheme, the category $\Qco(X)$ of quasi-coherent sheaves on $X$ is locally finitely presented (see \cite[I.6.9.12]{GD} or \cite[proposition 7]{Garkusha} for a precise formulation) and it comes equipped with a canonical tensor product, so one might wonder about the relationship (if any) between the two possible definitions of purity: the categorical one arising from the general fact that we are working with a locally finitely presented category, or the second one, arising from the usual tensor product in $\Qco(X)$. We shall denote by fp-pure the notion of purity in the first sense and just by pure in the second. Thus the first part of this paper is devoted to exploring the relation between the two notions on $\Qco(X)$. Actually we will consider a slightly different notion of purity in $\Qco(X)$. Namely, we say that a short exact sequence $\mathbb{E}$ in $\Qco(X)$ is pure exact provided that $\mathbb{E}\otimes \scrM$ is exact for each sheaf of $\O_X$-modules $\scrM$ (so not just the quasi-coherent ones). As we point out in the remark \ref{rem} this is equivalent to $\mathbb{E}\otimes -$ being exact in $\Qco(X)$ provided that $X$ is quasi-separated. The reason for considering this more general definition is that it is always equivalent to purity on the stalks (proposition \ref{pure1}). So this justifies its geometrical nature.

From this point of view, the present paper can be seen as a continuation on the ongoing program initiated in the work \cite{ES} where a wide class of projective schemes was exhibited that do not have nontrivial \emph{categorical flat} quasi--coherent sheaves (that is, quasi--coherent sheaves such that each short exact sequence ending on them is fp-pure). If we denote by $\Purefp$ and by $\Pure$ the classes of fp-pure and pure short exact sequences in $\Qco(X)$, we prove the following result (proposition \ref{pure3}):

\medskip\par\noindent
{\bf Proposition.} If $X$ is a concentrated scheme, $\Purefp \subseteq \Pure$.

\medskip\par
In particular this allows us to clarify the general relation between categorical and geometrical flatness for concentrated schemes (corollary \ref{relacion.flatfp.flattensor}):
\medskip\par\noindent
{\bf Proposition.}
Assume that $X$ is quasi--compact and semi--separated. Then each categorical flat sheaf in $\Qco(X)$ is geometrical flat.

\medskip\par
The converse is not true, in general, for non-affine schemes. This is one of the main results in \cite{ES} (\cite[theorem 4.4]{ES}).

Section 4 of the paper is devoted to showing that pure injective envelopes do exist with respect to both notions of purity.
The first proof is a particular instance of a theorem due to Herzog in \cite{Herzog} (see also \cite{Garcia}) on the existence of pure injective envelopes in locally finitely presented additive categories:

\medskip\par\noindent
{\bf Theorem.} Let $X$ be a concentrated scheme. Then every quasi-coherent sheaf in $\Qco(X)$ admits an fp-pure injective envelope, which is an fp-pure monomorphism.

\medskip\par
However, we can show that pure injective envelopes with respect to the geometrical purity always exist, without assuming any condition on the scheme (theorem \ref{tensor-pure.injenvelope}):

\medskip\par\noindent
{\bf Theorem.} Let $X$ be any scheme. Each quasi-coherent sheaf in $\Qco(X)$ has a pure injective envelope which is a pure monomorphism.

\medskip\par
In Section 5 we will focus on the geometrical pure notion in $\mathcal O_X\Mod$ and $\Qco(X)$ and we introduce the classes of (locally) absolutely pure sheaves of modules and quasi-coherent sheaves. Given an associative ring $R$ with unit, a left $R$-module is absolutely pure if every finite system of linear equations whose independent terms lie in $M$ possesses a solution in $M$. This is equivalent to saying that $M$ is a pure submodule of any $R$-module that contains it. In some aspects these behave like injective $R$-modules (see \cite{Fieldhouse,Maddox,Megibben,Prest2,Sfp} for a general treatment of absolutely pure modules and \cite{Pinzon} for a revisited study). In fact, Noetherian rings can be characterized in terms of properties of absolutely pure modules. Namely, $R$ is Noetherian if, and only if, the class of absolutely pure $R$-modules coincides with the class of injective $R$-modules (\cite{Megibben}). We will exhibit the main properties of (locally) absolutely pure sheaves of modules, both in $\Qco(X)$ and in $\mathcal O_X\Mod$, in case $X$ is locally coherent scheme. For instance we show in proposition \ref{ap3} that local absolutely purity in $\Qco(X)$ can be checked on a particular affine covering of $X$. And we also see that locally absolutely quasi-coherent sheaves are precisely the absolutely pure $\mathcal O_X$-modules that are quasi-coherent. This is analogous to the
question posted in \cite[II, \S 7, pg.135]{Hartshorne} for locally Noetherian schemes (cf. \cite[lemma 2.1.3]{Conrad}). Then we characterize locally Noetherian closed subschemes of the projective space ${\mathbb P}^n(A)$ in terms of its class of absolutely pure quasi-coherent sheaves:

\medskip\par\noindent
{\bf Theorem. } A locally coherent closed subscheme $X\subseteq \mathbb{P}^n(A)$ is locally Noetherian if, and only if, every locally absolutely pure quasi-coherent sheaf is locally injective.

\medskip\par
If $X$ is a Noetherian scheme, it is known that the class of locally injective quasi--coherent sheaves is a covering class in $\Qco(X)$. We finish this section by extending this result to the class of locally absolutely pure quasi-coherent sheaves on a locally coherent scheme $X$.

\medskip\par\noindent
{\bf Theorem. } Let $X$ be a locally coherent scheme. Then every quasi--coherent sheaf in $\Qco(X)$ admits a locally absolutely pure cover.

\bigskip\par\noindent

\section{Preliminaries}
In this work, all rings used will be commutative with identity.

Following \cite{CB}, an additive category $\mathcal A$ with direct limits is said to be \emph{locally finitely presented} provided that the skeleton of the subcategory of finitely presented objects in $\mathcal A$ is small, and each object of $\mathcal A$ is a direct limit of finitely presented objects. Here an object $A$ in $\mathcal A$ is called \emph{finitely presented} if the functor $\Hom_{\mathcal A}(A,-)$ preserves direct limits.

For instance, for any ring $S$ (not necessarily commutative and with unit) the category $S\Mod$ of left $S$-modules, is locally finitely presented \cite{L}. If $X$ has a basis of compact open sets then the category $\mathcal O_X\Mod$ of all sheaves of $\mathcal O_X$-modules, is locally finitely presented (see, for example, \cite[Theorem 5.6]{Prest} or \cite[Theorem 16.3.17]{Prest2}). If $X$ is a concentrated scheme (i.e. quasi--compact and quasi--separated) then the category $\Qco(X)$ of quasi--coherent sheaves of $\mathcal O_X$-modules is also locally finitely presented (see \cite[Proposition 7]{Garkusha} for a proof based on \cite[I.6.9.12]{GD}).

We recall that a short exact sequence of $R$-modules $0\to L\to M\to N\to 0$ is called \emph{pure} if $T\otimes L\to T\otimes M$ is a monomorphism for every $R$-module $T$. This is equivalent to $0\to\Hom_R(F,L)\to\Hom_R(F,M)\to \Hom_R(F,N)\to 0$ being exact, for each finitely presented $R$-module $M$. The last condition can be adapted to give the usual definition of a sequence  $0\to L\to M\to N\to 0$ of morphisms to be pure exact in an arbitrary locally finitely presented category $\mathcal A$.

\begin{definition}
Let $\mathcal{C}$ be a Grothendieck category. A  direct system of
objects of $\mathcal{C}$, $(M_\alpha \mid \alpha \leq \lambda)$, is
said to be a \emph{continuous system of monomorphisms} if $M_0=0$,
$M_\beta= \varinjlim _{\alpha< \beta} M_{\alpha}$ for each limit
ordinal $\beta \leq \lambda$ and all the morphisms in the system are
monomorphisms.

Let $\mathcal{S}$ be a class of objects which is closed under
isomorphisms. An object $M$ of $\mathcal{C}$ is said to be
\emph{$\mathcal{S}$-filtered} if there is a continuous system $(M_\alpha
\mid \alpha \leq \lambda)$ of subobjects of $M$ such that
$M=M_{\lambda}$ and $M_{\alpha+1}/M_{\alpha}$ is isomorphic  to an
object of $\mathcal{S}$ for each $\alpha< \lambda$.
\end{definition}

The class of $\mathcal{S}$-filtered objects in $\mathcal{C}$ is
denoted by $\Filt(\mathcal{S})$. The relation $\mathcal{S} \subseteq
\Filt(\mathcal{S})$ always holds. In the case of being
$\Filt(\mathcal{S}) \subseteq \mathcal{S}$, the class $\mathcal{S}$
is said to be \emph{closed under $\mathcal{S}$-filtrations.}

\begin{definition}
Let $\mathcal{F}$ be a class of objects of a Grothendieck category $\mathcal{A}$. A morphism
$\phi : F \rightarrow M$ of $\mathcal{C}$ is said to be an
\emph{$\mathcal{F}$-precover} of $M$ if $F \in \mathcal{F}$ and if
$\Hom(F',F)\rightarrow \Hom (F',M) \rightarrow 0$ is exact for every
$F' \in \mathcal{F}$. If any morphism $f:F \rightarrow F$ is such that
$ \phi \circ f= \phi$ is an isomorphism, then it is called an
\emph{$\mathcal{F}$-cover} of $M$. If the class $\mathcal{F}$ is such that
every object has an $\mathcal{F}$-cover, then $\mathcal{F}$ is
called a \emph{precovering class}. The dual notions are those of
$\mathcal{F}$-envelope and enveloping class.
\end{definition}

\section{Purity in $\Qco(X)$}
Let $X$ be a scheme and  $\scrF,\scrG$ be $\O_X$-modules. The tensor product $\scrF \otimes \scrG$ is defined as the sheafification of the presheaf $U \rightarrow \scrF (U) \otimes_{\O_X(U)} \scrG (U)$, for each open subset $U\subseteq X$. There is also an internal $\Hom$ functor in $\O_X\Mod$, $\ihom(-,-)$. The image $\ihom(\scrF,\scrG)(U)$ on an open subset $U \subseteq X$ is  $\Hom(\scrF|_U,\scrG|_U)$. It is known that the pair $(-\otimes-, \ihom(-,-) )$ makes $\O_X\Mod$ a closed symmetric monoidal category.
\begin{definition}\label{pure33}
Let $0 \rightarrow \scrF\stackrel{\tau}{\to} \scrG$ be an exact sequence  in $\O_X\Mod$. This is called  \emph{pure exact} if, for each $\scrM\in \O_X\Mod$, the induced sequence $$ 0\to \scrM\otimes \scrF\stackrel{id\otimes\tau}{\longrightarrow}\scrM\otimes \scrG$$ is exact.
\end{definition}
\begin{proposition}\label{pure1}
Let $0 \rightarrow \scrF\stackrel{\tau}{\to} \scrG$ be an exact sequence in $\O_X\Mod$. The following conditions are equivalent:
\begin{enumerate}
\item  the sequence is pure exact.
\item For each $x\in X$ the monomorphism $0 \rightarrow \scrF_x\stackrel{\tau}{\to} \scrG_x$ in $\O_{X,x}\Mod$, is pure.

\end{enumerate}
\end{proposition}
\begin{proof}
$1.\Rightarrow 2.$ Let $M\in \O_{X,x}\Mod$. Then  $i_{x,*}M$ (the skyscraper sheaf with respect to $M$) is an $\O_X$-module such that $(i_{x,*}M)_x=M$.
Since $0\to \scrF\stackrel{\tau}{\to} \scrG$ is pure, $$0\to i_{x,*}M\otimes \scrF\stackrel{}{\to} i_{x,*}M\otimes \scrG$$ is exact, that is, for each $x\in X$, $$0\to (i_{x,*}M\otimes \scrF)_x\stackrel{}{\to} (i_{x,*}M\otimes \scrG)_x$$ is exact in $\O_{X,x}\Mod$. But for each $\scrA\in \O_{X,x}\Mod$, $(i_{x,*}M\otimes \scrA)_x\cong M\otimes \scrA_x$. Hence, from the previous, we follow that $$0\to M\otimes \scrF_x\stackrel{}{\to} M\otimes \scrG_x $$ is exact in $\O_{X,x}\Mod$. So $0\to \scrF_x\stackrel{\tau}{\to} \scrG_x$ is pure.

\medskip\par\noindent
$2.\Rightarrow 1.$ Let  $0 \rightarrow \scrF\stackrel{\tau}{\to} \scrG$ be an exact sequence in $\O_X\Mod$ (so, for each $x\in X$, $0\to \scrF_x\stackrel{\tau_x}{\to} \scrG_x$ is exact in $\O_{X,x}\Mod$). Given $\scrM\in \O_{X}\Mod$, the induced $ \scrM\otimes \scrF\stackrel{id\otimes\tau}{\longrightarrow}\scrM\otimes \scrG$ will be a monomorphism if, and only if, for each $x\in X$ the morphism of $\O_{X,x}$-modules $(\scrM\otimes \scrF)_x\stackrel{(id\otimes\tau)_x}{\longrightarrow}(\scrM\otimes \scrG)_x$ is a monomorphism. But, for each $x\in X$, and $\scrA\in \O_X\Mod$, $(\scrM\otimes \scrA)_x\cong \scrM_x\otimes \scrA_x$. So by $2.$ it follows that $ \scrM\otimes \scrF\stackrel{id\otimes\tau}{\longrightarrow}\scrM\otimes \scrG$ is a monomorphism. Therefore $0\rightarrow \scrF\stackrel{\tau}{\to} \scrG$ is pure.
\end{proof}

\begin{proposition}\label{pure-allafine}
Let $X$ be a scheme and $\scrF,\scrG\in \Qco(X)$. The following conditions are equivalent:
\begin{enumerate}
\item $0\to \scrF\stackrel{\tau}{\to} \scrG$ is pure exact.
\item $0\to \scrF(U)\stackrel{\tau_U}{\longrightarrow} \scrG(U)$ is pure in $\O_X(U)\Mod$, for each open affine $U\subseteq X$.
\end{enumerate}
\end{proposition}
\begin{proof}
$1.\Rightarrow 2.$ Let  $U$ be an affine open subset of $X$ and $\imath:U\hookrightarrow X$ be the inclusion. And let $M\in \O_X(U)\Mod$. Then $\imath_*(\widetilde{M})$ is an $\O_X$-module. Therefore $$0\to \imath_*(\widetilde{M})\otimes \scrF\to \imath_*(\widetilde{M})\otimes \scrG  $$is exact. But then $$0\to (\imath_*(\widetilde{M})\otimes \scrF)(U)\to (\imath_*(\widetilde{M})\otimes \scrG)(U)  $$ is exact in $\O_X(U)\Mod$, that is, $$0\to \imath_*(\widetilde{M})(U)\otimes \scrF(U)\to \imath_*(\widetilde{M})(U)\otimes \scrG(U)  $$is exact. Since, for each $\O_X(U)$-module $A$, $\imath_*(\widetilde{A})(U)=A$, we get that $0\to M\otimes \scrF(U)\to M\otimes \scrG(U)$ is exact. Thus $0\to \scrF(U)\to \scrG(U)$ is pure.

\medskip\par\noindent
$2.\Rightarrow 1.$ This is immediate just by observing that, for each affine open set $U\subseteq X$, $(\scrF\otimes\scrG)(U)\cong \scrF(U)\otimes \scrG(U)$, and that a morphism $\tau$ in $\O_X\Mod$ is a monomorphism if, and only if, $\tau_U$ is a monomorphism in $\O_X(U)\Mod$.
\end{proof}

\begin{proposition}\label{pure2}
Let $X$ be a scheme and $\scrF, \scrG \in   \Qco(X)$. The following statements are equivalent:
\begin{enumerate}
\item $0\to \scrF\stackrel{\tau}{\to} \scrG$ is pure exact.
\item There exists an open covering of $X$ by affine open sets $\mathscr{U}=\{U_i\}$ such that $0\to \scrF(U_i)\stackrel{\tau_{U_i}}{\longrightarrow} \scrG(U_i)$ is pure in $\O_X(U_i)\Mod$.
\item $0\to \scrF_x\stackrel{\tau_x}{\to} \scrG_x$ is pure in $\O_{X,x}\Mod$, for each $x\in X$.

\end{enumerate}

\end{proposition}
\begin{proof}
$1.\Rightarrow 2.$ It follows from proposition \ref{pure-allafine}.

\medskip\par\noindent
$2.\Rightarrow 3.$ Let $x\in X$. Then there exists $U_i\in \mathscr{U}$ such that $x\in U_i=Spec(A_i)$, for some ring $A_i$. But then the claim follows by observing that $\scrF_x=(\widetilde{\scrF(U_i)})_x\cong  \widetilde{\scrF(U_i)_x}$ and noticing that if $0\to M\to N$ is pure exact in $A_i\Mod$, then $0\to M_x\to N_x$ is pure exact in $(A_i)_x\Mod$.

\medskip\par\noindent
$3.\Rightarrow 1.$ By proposition \ref{pure1}, we know that $\tau$ is pure in $\O_X\Mod$.
\end{proof}

\begin{remark}\label{rem}
Note that $\Qco(X)$ is a monoidal category with the tensor product induced  from $\O_X \Mod$. So we could also define a notion of purity in $\Qco(X)$ by using this monoidal structure, that is, $0\rightarrow \scrF \rightarrow \scrG$ is pure exact provided that it is $\scrM \otimes -$ exact, for each $\scrM \in \Qco(X)$. In case $X$ is quasi-separated this notion agrees with the one we have considered (that is, $0\rightarrow \scrF \rightarrow \scrG$ is pure in $\Qco(X)$ if it is pure in $\O_X\Mod$). This is because the direct image functor $\imath_*(\widetilde{M})$ preserves quasi--coherence when $X$ is quasi-separated in the proof of proposition \ref{pure-allafine}.


\end{remark}

Over an affine scheme $X$, the category of quasi-coherent sheaves on $X$ is equivalent to the category $\O_X(X) \Mod$. So the following lemma can be easily obtained.

\begin{lemma}
Let $\scrF \in \Qco(X)$ and $U$ be an affine open subset of $X$. Then $\scrF\mid_U$ is finitely presented in $\Qco(U)$  if and only if $\scrF(U)$ is finitely presented.
\end{lemma}

\begin{proposition}
Assume that $X$ is semi-separated or concentrated.
Let $\scrF \in \Qco(X)$ and consider the following assertions.
\begin{enumerate}
\item $\scrF$ is a finitely presented object in $\Qco(X)$.
\item $\scrF\mid_U$ is finitely presented in $\Qco(U)$ for all affine open subsets $U \subseteq X$.
\item $\scrF_x$ is finitely presented for each $x\in X$.

\end{enumerate}

Then the implications $1\Rightarrow 2\Rightarrow 3$ hold. If $X$ is concentrated, then $1\Leftrightarrow 2$ (\cite[proposition 75]{Mur2}).
\end{proposition}
 \begin{proof}
$1. \Rightarrow 2.$ We have to show that the
canonical morphism $$\psi: \varinjlim \Hom (\scrF|_U,\widetilde{B_i})
\rightarrow \Hom(\scrF|_U,\varinjlim \widetilde{B_i})$$ is an isomorphism
for any direct system $\{\widetilde{B_i},\varphi_{ij}\}_I$ of
quasi-coherent $\O_X |_U$-modules. We have the following commutative diagram
$$\xymatrix{\Hom(\scrF|_U, \widetilde{B_i}) \ar[r]\ar[d]& \varinjlim \Hom (\scrF|_U,\widetilde{B_i})\ar[d]\ar[r]^\psi & \Hom(\scrF|_U,\varinjlim \widetilde{B_i})\ar[d]\\
\Hom(\scrF,\imath_*( \widetilde{B_i}))\ar[r] & \varinjlim\Hom(\scrF,\imath_*
(\widetilde{B_i}))\ar[r]^{\psi '}& \Hom(\scrF,\varinjlim \imath_*(
\widetilde{B_i})) }.$$
Columns
are isomorphisms because of the adjoint pair $(\res_U, \imath_*)$. For the third column we also need to observe that, under the hypothesis on $X$, the direct image functor $\imath_*$ preserves direct limits. Since $\scrF$ is finitely presented, the
canonical morphism $\psi '$ is an isomorphism. So $\psi$ is an
isomorphism.


\medskip\par\noindent
$3. \Rightarrow 4.$ $\scrF_x \cong M_p$ for some
finitely presented $R$-module $M$ and prime ideal $p$. Then $4.$ follows because the localization
of a finitely presented $R$-module is a finitely presented $R_p$-module.
\end{proof}

\begin{definition}(cf. \cite[\S 3]{CB})
An exact sequence $0 \rightarrow \scrF \rightarrow \scrG \rightarrow \scrT
\rightarrow 0$ in $\Qco(X)$ is called \emph{categorical pure}  if the functor  $\Hom(\scrH,-)$ leaves the sequence exact for every finitely
presented quasi-coherent $\O_X$-module $\scrH$.
\end{definition}
We shall denote by $\Purefp$ the class of  categorical pure short exact sequences in $\Qco(X)$   and by $\Pure$ the class of pure short exact sequences in $\Qco(X)$, as in proposition \ref{pure2}.
\begin{proposition}\label{pure3}
If $\Qco(X)$ is a locally finitely presented category then categorical pure short exact sequences are pure exact, that is, $\Purefp \subseteq \Pure$.
\end{proposition}
 \begin{proof}
Let
$\mathbb{E} \equiv 0\rightarrow \scrF \rightarrow \scrG\rightarrow
\scrH\rightarrow 0$
be an exact sequence  in $\Purefp$. By assumption,
$\scrH=\varinjlim \scrH_i$ where $\scrH_i$ is a finitely presented object in
$\Qco(X)$ for each $i$. Now, for each $i$, the top row of the following pullback
diagram,
$$\mathbb{E}^i=\xymatrix{0 \ar[r] & \scrF  \ar[r]\ar[d]& \scrG _i \ar[r]\ar[d]& \scrH_i \ar[d]\ar[r]& 0\\
0\ar[r]& \scrF \ar[r] & \scrG \ar[r]& \scrH\ar[r] & 0}$$ is  a categorical pure exact sequence
ending with a  finitely presented object $\scrH_i$. Therefore, each
$\mathbb{E}^i$ splits for every $i$. That is, $\mathbb{E}=\varinjlim \mathbb{E}^i$ where $\mathbb{E}^i$
is a splitting exact sequence for every $i$. Now taking the stalk at $x\in X$, we get $\mathbb{E}_x=\varinjlim \mathbb{E}^i_x$. Then $\mathbb{E}^i_x$ is pure exact in $\O_{X,x}\Mod$ for each $x\in X$, and so is $\mathbb{E}_x$. Hence, by proposition \ref{pure2}, $\mathbb{E}$ is a pure exact sequence in $\Qco(X)$.
\end{proof}

We recall that a quasi-coherent $\O_X$-module $\scrF$ is \emph{flat} if $\scrF \otimes -$ is exact in $\O_X \Mod$. Equivalently, $\scrF (U)$ is flat as an $\O_X(U)$-module for each affine open subset $U \subseteq X$, or $\scrF_x$ is flat as an $\O_{X,x}$-module for each $x\in X$. We will denote by $\Flat$ the class of all flat quasi-coherent sheaves.
\begin{definition}
A quasi-coherent $\O_X$-module $\scrF$ is called {\it tensor flat} (resp. {\it fp-flat}) if every short exact
sequence in  $\Qco(X)$ ending in $\scrF$ is pure exact (resp. is categorical pure). We shall denote by $\Flatt$ (resp. by $\Flatfp$) the class of all tensor flat quasi-coherent sheaves (resp. the class of all fp-flat quasi-coherent sheaves).

\end{definition}
\begin{proposition}\label{p.}
Let $\scrF\in \Qco(X)$. If $\scrF$ is flat, then  it is also tensor flat. In case $X$ is semi-separated, the converse also holds.
\end{proposition}
\begin{proof}
Let $0\to \scrT\to\scrG\to \scrF\to 0$ be an exact sequence in $\Qco(X)$. Given an affine open $U\subseteq X$, $0\to \scrT(U)\to\scrG(U)\to \scrF(U)\to 0$ is also exact. Since $\scrF(U)$ is a flat $\O_X(U)$-module, we deduce from proposition \ref{pure-allafine} that $\scrF$ is tensor flat.

If $X$ is a semi-separated scheme then the direct image functor $\imath_*$ for the inclusion map $\imath : U \hookrightarrow X$, where $U$ is affine, is exact.

Let $\scrF\in \Qco(X)$ be tensor flat. We need to show that $\scrF(U)$ is a flat $\O_X(U)$-module, for each
affine open subset $U\subseteq X$. Let $$\xymatrix{0\ar[r]& A \ar[r] &
B\ar[r]& \scrF(U)\ar[r] & 0}$$ be an exact sequence of $\O(U)$-modules.
By the previous observation, we have an exact sequence
$$\xymatrix{0\ar[r]& \imath_*(\widetilde{A}) \ar[r] &\imath_*(\widetilde{B}) \ar[r]&
\imath_*(\widetilde{\scrF|_U})\ar[r] & 0}.$$ If we take the pullback of the morphism $\imath_*(\widetilde{B}) \to
\imath_*(\widetilde{\scrF|_U}) $ and the canonical morphism $\scrF \rightarrow
\imath_*(\widetilde{\scrF|_U})$, we get the commutative diagram with exact rows:
$$\xymatrix{0 \ar[r] & \imath_*(\widetilde{A}) \ar[r]\ar[d]& \scrH \ar[r]\ar[d]& \scrF\ar[d]\ar[r]& 0\\
0\ar[r]& \imath_*(\widetilde{A}) \ar[r] &\imath_*(\widetilde{B})
\ar[r]& \imath_*(\widetilde{\scrF|_U})\ar[r] & 0}.$$
Since $\imath_*(\widetilde{A})$ and $\scrF$ are quasi-coherent, $\scrH$ is
quasi-coherent. By assumption, the first row is pure exact, so by proposition \ref{pure-allafine}
each image under affine open subset is pure exact. From this
diagram, it can be deduced that $\scrH(U) \cong
\imath_*(\widetilde{B})(U)=B$. So the short exact sequence  $\xymatrix{0\ar[r]& A \ar[r] &
B\ar[r]& \scrF(U)\ar[r] & 0}$ is pure and then $\scrF(U)$ is a flat $\O_X(U)$-module.
\end{proof}
\begin{corollary}\label{relacion.flatfp.flattensor}
Assume that $\Qco(X)$ is locally finitely presented (for instance if $X$ is concentrated). Then $\Flatfp\subseteq\Flatt$. If $X$ is semi-separated then $\Flatfp\subseteq\Flatt=\Flat$.
\end{corollary}
\begin{proof}
This follows from  proposition \ref{pure3} and proposition \ref{p.}.
\end{proof}
\begin{remark}
The inclusions in corollary \ref{relacion.flatfp.flattensor} are strict. Namely in \cite[corollary 4.6]{ES} it is shown that $\Flatfp=0$ in case $X={\bf P}^n(R)$. In general there is a large class of projective schemes $X$ such that $\Flatfp=0$ in $\Qco(X)$ (see \cite[theorem 4.4]{ES}).

\end{remark}

\section{Pure injective envelopes}
\begin{definition}
A quasi-coherent $\O_X$-module $\scrM$ is said to be {\it fp-pure injective} (resp. {\it pure injective}) if for every short exact sequence $0\to \scrF\to \scrG\to \scrH\to 0$ in $\Purefp$ (resp. in $\Pure$) the sequence
$0\to \Hom(\scrH,\scrM)\to \Hom(\scrG,\scrM)\to \Hom(\scrF,\scrM)\to 0$ is exact. We shall denote by $\Pinjfp$ (resp. by $\Pinj$) the class of all fp-pure injective quasi-coherent sheaves (resp. the class of all pure injective quasi-coherent sheaves). In general, when we say that an $\O_X$-module is pure injective, we mean that it is `injective' with respect to all pure exact sequences in $\O_X\Mod$.

\end{definition}
\begin{remark}
\medskip\par\noindent
\begin{itemize}
\item If $X$ is concentrated then, by proposition \ref{pure3}, $\Pinj\subseteq \Pinjfp$.
\item Clearly, every injective quasi-coherent $\O_X$-module is both fp-pure injective and pure injective.
\end{itemize}
\end{remark}
\begin{theorem}
Let $X$ be a concentrated scheme. Then every $\scrM\in\Qco(X)$ admits an fp-pure injective envelope $\eta:\scrM\to \mathrm{PE}_{fp}(\scrM)$. That is, $\Pinjfp$ is enveloping.

Moreover the induced short exact sequence $$0\to \scrM \stackrel{\eta}{\longrightarrow} \mathrm{PE}_{fp}(\scrM)\longrightarrow \frac{\mathrm{PE}_{fp}(\scrM)}{\scrM}\to 0$$ is in $\Purefp$.
\end{theorem}
\begin{proof}
Since $X$ is concentrated, $\Qco(X)$ is a locally finitely presented Gro\-then\-dieck category. So the result follows from \cite[theorem 6]{Herzog} (see also \cite{Garcia}).
\end{proof}

Now we will recall the definition of an internal $\Hom$ functor in $\Qco(X)$ ($X$ is an arbitrary scheme). The category $\Qco(X)$ is Grothendieck abelian (see \cite[corollary 3.5]{EE} for the existence of a generator for $\Qco(X)$) and the inclusion functor $\Qco(X)\to \O_X\Mod$ has a right adjoint functor $C$ by the Special Adjoint Functor Theorem. This right adjoint functor is known in the literature as the coherator. The internal $\Hom$ functor is thus defined as
$\iHom(\scrF,\scrG)=C\mathscr{H}om(\scrF,\scrG)$, where $\mathscr{H}om(-,-)$ is the usual sheafhom functor. Therefore $\Qco(X)$ is a closed symmetric monoidal category with the usual tensor product and the $\iHom(-,-)$ a bifunctor, and there is a natural isomorphism $$\Hom(\scrF\otimes\scrG,\scrH)\cong \Hom(\scrF,\iHom(\scrG,\scrH)). $$ The unit object of the monoidal structure is given by $\O_X$. Thus one gets a natural equivalence
$\Hom(\O_X,\iHom(-,-))\simeq \Hom(-,-)$ so for each $\scrF,\scrG\in \Qco(X)$, there is a bijection $\Hom(F,G)\cong \iHom(\scrF,\scrG)(X)$.

Now since $\O_X\Mod$ is a Grothendieck category, it has injective envelopes. Let $\Lambda=\{\scrST_i:\ i\in I\}$ be a set of generators for $\O_X \Mod$ (see for example \cite[corollary 6.8]{Swan}). We pick an injective embedding $$\bigoplus_{\Lambda,\scrT} \scrST_i/\scrT\longrightarrow \scrE$$ where $\scrE\in\O_X \Mod$ is  injective and the sum runs also over all  $\O_X$-submodules of each $\scrST_i$. Then it is clear that such $\scrE$ is  an {\it injective cogenerator} for $\O_X \Mod$. This is an injective $\O_X$-module with the property that for every nonzero $\scrG \in \O_X \Mod$ there exists a nonzero morphism $\scrG\to\scrE$. Note that $C(\scrE)$ is an injective cogenerator in $\Qco(X)$. Indeed,  the inclusion functor $\Qco(X) \to \O_X\Mod$ is an exact functor  with right adjoint $C$.

We shall denote by $\scrM^{\vee}$ the {\it character  $\O_X$-module} given by $\scrM^{\vee}=\mathscr{H}om(\scrM,\scrE)$. There is a canonical map $ev:\scrM \to \scrM^{\vee \vee}$.
\begin{proposition}\label{espur}
Given $\scrM\in\O_X \Mod$, the character  $\O_X$-module $\scrM^{\vee}$ is pure injective in $\O_X \Mod$.
\end{proposition}
\begin{proof}
Let $0\to \scrT\to \scrN\to \scrH\to 0$ be a pure exact sequence in $\O_X \Mod$. Then $$\Hom(\scrN,\scrM^\vee)\to \Hom(\scrT,\scrM^\vee)\to 0$$ is exact if and only if $$\Hom(\scrN\otimes\scrM,\scrE)\to \Hom(\scrT\otimes\scrM,\scrE)\to 0$$ is exact. But the latter follows since $0\to \scrT\otimes\scrM\to \scrN\otimes\scrM$ is exact and $\scrE$ is an injective cogenerator.
\end{proof}

\begin{proposition}\label{esci}
A short exact sequence in $\O_X \Mod$,
$$0\to \scrF\longrightarrow \scrG\longrightarrow \scrT\to 0$$ is pure exact if and only if $$0\to \scrT^\vee\longrightarrow \scrG^\vee \longrightarrow \scrF^\vee\to 0$$ splits.
\end{proposition}
\begin{proof}
The proof is the same as that in categories of modules (see for example \cite[proposition 5.3.8]{EdO}). It is necessary to point out that in any Grothendieck category $\mathcal C$, with an injective cogenerator $E$, a sequence $0\to M\to L\to N\to 0$ is exact if, and only if, $0\to \Hom(N,E)\to \Hom(L,E)\to \Hom(M,E)\to 0$ is exact.
\end{proof}

\begin{corollary}\label{preen}
For any $\scrM\in \O_X \Mod$ the evaluation map $ev:\scrM\to \scrM^{\vee\vee}$ is a pure monomorphism. 
\end{corollary}
\begin{proof}
First we will see that $ev$ is injective. Let $0\neq x\in \scrM(U)$ for some affine open $U$. Then there exists a nonzero $\O_X$-module $\scrST/\scrT\subseteq \scrM$, where $\scrST\in \Lambda$, with $x\in\scrST/\scrT$. By the definition of $\scrE$, there is a monomorphism $\alpha:\scrST/\scrT\to \scrE$ with $\alpha(x)\neq 0$. Then $\alpha$ extends to $\alpha':\scrM\to \scrE$. And $ev(x)(\alpha')=\alpha'(x)\neq 0$. So we are done. To show that $ev:\scrM\to \scrM^{\vee \vee}$ is pure exact we need to show, by proposition \ref{esci}, that $\scrM^{\vee \vee \vee}\to \scrM^{\vee}$ admits a section, but $ev^{\vee}:\scrM^{\vee}\to \scrM^{\vee\vee\vee}$ is a such section.
\end{proof}

\begin{lemma}\label{preen1}
Let $\scrM$ be a pure-injective  $\O_X$-module. Then its coherator $C(\scrM)$ is pure injective in $\Qco(X)$, as well.
\end{lemma}
\begin{proof}
Let $0 \rightarrow \scrF \rightarrow \scrG$ be a pure exact sequence in $\Qco(X)$. This means that it is pure exact in $\O_X\Mod$. So we have an exact sequence
$$\Hom_{\O_X\Mod}(\scrG, \scrM) \rightarrow \Hom_{\O_X\Mod}(\scrF,\scrM)\rightarrow 0.$$
Since $(\iota, C)$ is an adjoint pair where $\iota: \Qco(X) \hookrightarrow \O_X \Mod$, this implies that
$$\Hom_{\Qco(X)}(\scrG, C(\scrM)) \rightarrow \Hom_{\Qco(X)}(\scrF,C(\scrM))\rightarrow 0$$
is exact.
\end{proof}

\begin{corollary}\label{preen2}
Every quasi-coherent sheaf $\scrM$ can be purely embedded into a pure injective quasi-coherent  sheaf. In particular, the class of pure injective quasi-coherent sheaves is preenveloping.
\end{corollary}
\begin{proof}
Let $\scrM$ be a quasi-coherent sheaf. By corollary \ref{preen}, there is a pure monomorphism $ev: \scrM \rightarrow \scrM^{\vee \vee}$, where $\scrM^{\vee \vee}$ is  a pure injective $\O_X$-module. So we apply the coherator functor on $\scrM^{\vee \vee}$, $C(\scrM^{\vee \vee})$. By lemma \ref{preen1}, it is a  pure injective quasi-coherent sheaf. The adjoint pair $(\iota,C)$  allows to factorize $ev$ over $C(\scrM^{\vee \vee})$. Indeed, $\Qco(X)$ is a coreflective subcategory of $\O_X \Mod$  and $\scrM$ is quasi-coherent. So there is a unique morphism $\varphi:\scrM \rightarrow C(\scrM^{\vee \vee})$ over which $ev$ is factorized. Then $\varphi$ is a pure monomorphism, as well.
\end{proof}

In order to show that the class $\Pure$ in $\Qco(X)$ is enveloping,  we will apply \cite[theorem 2.3.8]{Xu} (this, in turn, uses \cite[theorem 2.2.6]{Xu}). The arguments in these proofs are categorical and can be easily extended to our setup in $\Qco(X)$ by taking into account the following lemma:
\begin{lemma}\label{lema1}
For a given $\scrM\in \Qco(X)$, the class of sequences in $\Pure$ of the form $$0\to \scrM\to \mathscr{L}\to \scrT\to 0,$$ varying $\mathscr{L},\scrT\in \Qco(X)$ is closed under direct limits.
\end{lemma}

\begin{proof}
The argument is local and so it can be deduced from the corresponding result on module categories (see for example \cite[proposition 2.3.7]{Xu}).
\end{proof}


Combining lemma \ref{lema1} and corollary \ref{preen2} and applying the analogue to \cite[theorem 2.3.8]{Xu} for the category $\Qco(X)$, we get
\begin{theorem}\label{tensor-pure.injenvelope}
Every $\scrM\in\Qco(X)$ admits a pure injective envelope $\eta:\scrM\to \mathrm{PE}(\scrM)$. That is, $\Pinj$ is enveloping.

Moreover the induced short exact sequence $$0\to \scrM \stackrel{\eta}{\longrightarrow} \mathrm{PE}(\scrM)\longrightarrow \frac{\mathrm{PE}(\scrM)}{\scrM}\to 0$$ is in $\Pure$.
\end{theorem}

\section{Locally absolutely pure quasi-coherent shea\-ves and absolutely pure sheaves}

An $R$-module $A$ is \emph{absolutely pure} (\cite{Maddox}) if it is pure in every module containing it as a submodule. Absolutely pure modules are also studied with the terminology of FP-injectives (\cite{Sfp}). It follows immediately from the definition that $A$ is absolutely pure if, and only if, it is a pure submodule of some injective module. And therefore $A$ is absolutely pure if, and only if, $\Ext^1_R(M,A)=0$ for each finitely presented $R$-module $M$.

In this section we will study (locally) absolutely pure sheaves in both $\O_X\Mod$ and in $\Qco(X)$. Since we have pure exact sequences in categories of sheaves rather than categorical ones, we deal with tensor-purity to define absolutely pure sheaves in $\O_X\Mod$ and in $\Qco(X)$.

\begin{definition}
Let $(X,\O_X)$ be a scheme.
\begin{enumerate}
\item Let $\scrF$ be in $\O_X \Mod$. $\scrF$ is \emph{absolutely pure} in $\O_X \Mod$ if  every exact sequence $0 \rightarrow \scrF \rightarrow \scrG$ in $\O_X \Mod$ is pure exact in $\O_X \Mod$.
\item Let $\scrF$ be a quasi-coherent sheaf on $X$. $\scrF$ is called \emph{absolutely pure} in $\Qco(X)$ if every exact sequence $0 \rightarrow \scrF \rightarrow \scrG$ in $\Qco(X)$ is pure exact.
\item Let $\scrF$ be a quasi-coherent sheaf on $X$. $\scrF$ is called \emph{locally absolutely pure} if $\scrF(U)$ is absolutely pure over $\O_X(U)$ for every affine open $U \subseteq X$.

\end{enumerate}
\end{definition}
\begin{lemma}
All these notions of locally absolutely purity of quasi-coherent sheaves
and  absolutely purity in $\O_X \Mod$ and in $\Qco(X)$ are closed under
taking pure subobjects.
\end{lemma}
\begin{proof}
It follows from the fact that if $f\circ g$ is a pure monomorphism
with monomorphisms $f$ and $g$, then $g$ is a pure monomorphism.
\end{proof}

\begin{lemma}\label{ap4}
Let $\scrF$ be an $\O_X$-module. The following are equivalent:
\begin{enumerate}
\item $\scrF$ is absolutely pure in $\O_X \Mod$.
\item $\scrF \mid_{U_i}$ is absolutely pure in $\O_X \mid_{U_i} \Mod$ for a cover $\{U_i\}$ of $X$.
\end{enumerate}
\end{lemma}
\begin{proof}

$1.\Rightarrow 2.$ Let $U \subseteq X$ be open. Then the extension
of $\scrF\mid_U$  by zero outside $U$, $j_!(\scrF \mid_U)$,  is contained
in $\scrF$. Since the stalk of $j_!(\scrF \mid_U)$ is $\scrF_x$ if $x \in U$
and $0$  otherwise, $j_!(\scrF \mid_U)$ is a pure subsheaf of $\scrF$ in
$\O_X \Mod$. So $j_!(\scrF \mid_U)$ is absolutely pure in $\O_X \Mod$,
too.

Now let $\scrG$ be any $\O_X \mid_U$-module with an exact sequence $0
\rightarrow \scrF \mid_U \rightarrow \scrG$. Then $0 \rightarrow j_!(\scrF
\mid_U) \rightarrow j_!(\scrG)$ is still exact in $\O_X \Mod$. So it is
pure in $\O_X$. But this means that  $0 \rightarrow [j_!(\scrF
\mid_U)]_x \rightarrow [j_!(\scrG)]_x$ is pure for all $x \in X$.  For
$x \in U$, that exact sequence is equal to the exact sequence $0
\rightarrow (\scrF \mid_U)_x \rightarrow (\scrG)_x$ and $j_!(\scrF \mid_U)
\mid_U=\scrF \mid_U$ and $(j_!(\scrG))\mid_U=\scrG$. That proves the desired
implication.

\medskip\par\noindent
$2.\Rightarrow 1.$ Let $0 \rightarrow \scrF \rightarrow \scrG$ be an
exact sequence in $\O_X \Mod$. In order to show that it is pure
exact, we need to show that the morphism induced on the stalk is
pure exact, for every $x\in X$. But the restriction functor to open
subsets is left exact and $(\scrF \mid_U)_x = \scrF _x$. So the claim
follows.
\end{proof}

\begin{lemma}\label{ap5}
Let $\scrF$ be an $\O_X$-module. If $\scrF_x$ is absolutely pure for all
$x\in X$ then $\scrF$ is absolutely pure in $\O_X \Mod$.
\end{lemma}
\begin{proof}
Let $0 \rightarrow \scrF \rightarrow \scrG$ be an exact sequence in $\O_X
\Mod$. To be pure in $\O_X \Mod$ is equivalent to be pure at the
induced morphism on the stalk for every $x\in X$. So that proves our
implication.
\end{proof}

Let $X=Spec(R)$ be an affine scheme. The next proposition shows that in order to check that a quasi--coherent $\O_X$-module $\widetilde{A}$ is absolutely pure, it suffices that its restrictions $\widetilde{A}|_{D(s_i)}$, $i=1,\ldots, n$, be absolutely pure, where $\cup_{i=1}^n D(s_i)=X$, and $s_1,\ldots,s_n\in R$.

\begin{proposition}\label{ab_1}
Let $R$ be a ring and
$s_1,s_2, \ldots, s_n$ a finite number of elements of $R$ which generate the unit ideal. Let $A$ be an $R$-module. If $A_{s_i}$ is absolutely pure  over
$R_{s_i}$ for every $i=1,\ldots, n$ then $A$ is absolutely pure over
$R$.
\end{proposition}
\begin{proof}
Given $A \subseteq B$, we want to prove the canonical morphism $M
\otimes A \rightarrow M\otimes B$ is injective for every module $M$.
Let $K= \Ker(M\otimes A \rightarrow M\otimes B)$. Then by our
hypothesis we get $K_{s_i}=0$ for each $i=1,\ldots,n$. So if $x\in
K$, then ${s_i}^{h_i} x =0$  for some $h_i\geq 0$. But the set
$\{s_1,s_2, \ldots, s_n\}$ generates  $R$.  So we have $s_1t_1 +
\ldots + s_n t_n=1$ for some $t_1,\ldots, t_n \in R$. And also
$(s_1t_1 + \ldots + s_n t_n)^h x=0$ if $h > h_1+\ldots + h_n-1$,
i.e., $x=1.x=1^h x=0$.
\end{proof}

Let $X=Spec(R)$ be an affine scheme. Now we will see that in order to check that a quasi--coherent $\O_X$-module $\widetilde{A}$ is absolutely pure, it suffices to check that, for each $P\in X$, each stalk $\widetilde{M}_P$ is an absolutely pure $\O_{X,P}$-module.

\begin{proposition}\label{ab2}
If $A_P$ is absolutely pure over $R_P$ for every prime ideal $P$ then $A$ is absolutely pure over $R$.
\end{proposition}
\begin{proof}
Let $M$ be a finitely presented $R$-module. We want to prove that \\ $\Ext_R^1(M,A)=0.$ Since $M$ is finitely presented,
$$(\Ext_R^1(M,A))_P\cong \Ext_{R_P}^1(M_P,A_P)=0.$$
Since this is true for each prime ideal $P$, $\Ext_R^1(M,A)=0$. So $A$ is absolutely pure.
\end{proof}
Neither Propositions \ref{ab_1} and \ref{ab2} do not assume any condition on the ring $R$. Their converses are not true in general. However they are if
$R$ is coherent, see \cite[theorem 3.21]{Pinzon}. So it makes sense to
define a notion of locally absolutely pure quasi-coherent sheaves
over a locally coherent scheme. A scheme $(X,\O_X)$ is
\emph{locally coherent} provided that $\O_X(U)$ is a coherent ring, for
each affine open subset $U \subseteq X$. Since coherence descends along faithfully flat morphisms of rings (see \cite[corollary 2.1]{Harris}), it follows that $X$ is locally coherent if, and only if, $\O_X(U_i)$ is coherent for each $i\in I$ of some affine open covering $\{U_I\}_{i\in I}$ of $X$. So over a locally coherent scheme,
the next proposition states that in order to prove whether a
quasi-coherent sheaf is locally absolutely pure, it is sufficient
to look at  some cover by affine subsets of $X$. And these show that locally
absolute purity is a stalkwise property.
\begin{proposition}\label{ap3}
Let $(X,\O_X)$ be a locally coherent scheme. Then the following conditions are equivalent for a quasi-coherent sheaf $\scrF$:
\begin{enumerate}
\item $\scrF(U)$ is absolutely pure for every affine $U$.
\item $\scrF(U_i)$ is absolutely pure  for all $i \in I$ for some cover $\{U_i\}_{i\in I}$ of affine open subsets.
\item $\scrF_x$ is absolutely pure for all $x \in X$.
\end{enumerate}
\end{proposition}
\begin{proof}
We just need to prove the implications $(2 \Rightarrow 3)$ and $(3\Rightarrow 1)$. By \cite[theorem 3.21]{Pinzon} the localization of
an absolutely pure module over a coherent ring is again absolutely pure, so
 the first implication follows. For the second, Let $\scrF(U) \cong M$ for an $\O_X(U)$-module $M$. By assumption, $\scrF(U)_P \cong M_P$ is absolutely pure
 for all prime ideal $P$ of $\O_X(U)$. Hence,  $\scrF(U) = M$ is also absolutely pure by proposition \ref{ab2}.
\end{proof}

The next lemma shows that the locally absolutely pure objects in
$\Qco(X)$ on a locally coherent scheme $X$ are exactly the
absolutely pure $\O_X$-modules which  are quasi-coherent.

\begin{lemma}
Let $X$ be a locally coherent scheme and $\scrF$ be a quasi-coherent sheaf. Then $\scrF$ is locally absolutely
pure if and only if $\scrF$ is absolutely pure in $\O_X \Mod$.
\end{lemma}
\begin{proof}
It follows by lemma \ref{ap4} and proposition \ref{ap3}.
\end{proof}
 At this
point, we may consider  the relation between absolutely pure
quasi-coherent sheaves and locally absolutely pure quasi-coherent
sheaves.
\begin{lemma}\label{ap6}
Let $X$ be a locally coherent scheme. Every locally absolutely pure quasi-coherent sheaf is absolutely
pure in $\Qco(X)$.
\end{lemma}
\begin{proof}
This follows from proposition \ref{ap3} and proposition \ref{pure2}.
\end{proof}

The converse of lemma \ref{ap6} is not clear in general. But it is
true if $X=Spec(R)$ is affine and $R$ is coherent, or if $X$ is locally Noetherian. The first case is
clear since $\Qco(X) \cong \O_X(X)\Mod$. For the second, let $\scrF$ be
absolutely pure in $\Qco(X)$ and $E(\scrF)$ be its injective
envelope in $\Qco(X)$. Then $0\rightarrow \scrF \rightarrow E(\scrF)$ is
pure exact. So, for each affine open subset $U \subseteq X$,
$0\rightarrow \scrF(U) \rightarrow E(\scrF)(U)$ is pure exact in $\O_X(U)
\Mod$. But $E(\scrF)(U)$ is an injective  $\O_X(U)$-module and $\scrF(U)$
is a pure submodule of it. Hence $\scrF(U)$ is absolutely pure, for
each affine $U \subseteq X$. So, $\scrF$ is a locally absolutely pure
quasi-coherent sheaf.

\begin{proposition}
Let $X$ be a locally coherent scheme. If the class of injective sheaves in  $\O_X \Mod$ is equal to the
class  of absolutely pure sheaves in $\O_X \Mod$, then $X$ is a
locally Noetherian scheme.
\end{proposition}
\begin{proof}
Suppose that these classes are equal. Let $M$ be an absolutely pure
$\O_X(U)$-module where $U$ is an affine open subset. Then the sheaf
$j_! (\widetilde{M})$ obtained by extending $\widetilde{M}$ by zero
outside $U$ is absolutely pure $\O_X$-module by lemma \ref{ap5}. By
assumption, it is injective in $\O_X \Mod$. So, its restriction
$(j_! (\widetilde{M})) \mid_U = \widetilde{M}$ is injective in $\O_X
\mid_U \Mod$. Since $\widetilde{M}$ is quasi-coherent, it is
injective in $\Qco(U)$ which implies that $M$ is injective
$\O_X(U)$-module. So $\O_X(U)$ is Noetherian ring and $X$ is a
locally Noetherian scheme.
\end{proof}

Now we will extend the known fact that a ring $R$ is
Noetherian if and only if each absolutely pure $R$-module  is injective (see \cite[theorem 3]{Megibben}) for closed subschemes of
$\mathbb{P}^n(R)$ which are locally coherent. Let $R$ be a commutative
ring and $X=\mathbb{P}^n(R)$ be a projective scheme over $R$, where
$n\in \mathbb{N}$. Then take a cover of $X$ consisting of affine open
subsets $D_+(x_i)$  for all $i=0,\ldots n$, and
 all possible intersections. In this case, our cover contains  basic open subsets of this form
$$D_+( \prod_{i \in v}x_i),$$ where $v \subseteq \{0,1, \ldots ,n
\}$. It is known that the category of quasi-coherent sheaves over a
scheme is equivalent to the class of certain module
representations over some quiver satisfying the cocycle condition, see \cite{EE}. In our case, the
vertices of our quiver are all subsets of $\{0,1,\ldots,n\}$ and we
have only one edge $v\rightarrow w$ for each $v\subseteq w \subseteq
\{0,1,\ldots,n\}$ since $D_+(\prod_{i \in w}x_i) \subseteq
D_+(\prod_{i \in v}x_i)$. Its ring representation has
$$\O_{\mathbb{P}^n(R)}(D_+(\prod_{i \in v}x_i)) =R[x_0,\ldots,x_n]_{(\prod_{i \in v}x_i)}$$
on each vertex $v$, which is the subring of the localization
$R[x_0,\ldots,x_n]_{\prod_{i \in v}x_i}$ containing its degree zero
elements. It is isomorphic to the polynomial ring  on the ring $R$
with the variables $\frac{x_j}{x_i}$ where $j=0, \ldots ,n$ and $i
\in v$. We denote this polynomial ring by $R[v]$. Then the
representation $\scrR$ with respect to this quiver with relations is defined as
$\scrR(v)=R[v]$, for each vertex $v$ and there is an edge $\scrR(v) \hookrightarrow \scrR(w)$ provided that $v
\subseteq w$. Finally, a quasi--coherent sheaf $\scrM$ on $\Qco(X)$ is uniquely determined by a compatible family of $\scrR(v)$-modules $\scrM(v)$, satisfying that $$S_{vw}^{-1} f_{vw} :S_{vw}^{-1}\scrM(v)
\longrightarrow S_{vw}^{-1}\scrM(w)=\scrM(w)$$ is an isomorphism as
$R[w]$-modules for each $f_{vw}:\scrM(v) \rightarrow \scrM(w)$ where
$S_{vw}$ is the multiplicative set generated by the
$\{x_j/x_i|\textrm{ }  j \in w \setminus v ,\textrm{ }i \in v\} \cup
\{1\}$ and $v\subset w$.

Recall that a closed subscheme $X$ of $\mathbb{P}^n(R) $ is given
by a quasi-coherent sheaf of ideals, i.e. we have an ideal  $ I_v
\subseteq R[v] $ for each $v$ with  $R[w]\otimes_{R[v]} I_v \cong
I_w $ when $v \subseteq w$. This means $I_v \rightarrow I_w $ is the
localization by the same multiplicative set as above. But then
$R[v]/I_v \rightarrow R[w]/I_w$ is also a localization. So, by abusing the notation, we shall also denote by $\scrR$ the structural sheaf of rings attached to $X$.

\begin{proposition}
A closed subcheme $X \subseteq \mathbb{P}^n(R) $ which is locally coherent (for instance if $X=\mathbb{P}^n(R)$ and $R$ is stably coherent) is locally Noetherian if and only if locally absolutely pure quasi-coherent sheaves are locally
injective.
\end{proposition}
\begin{proof}

``If'' part is clear. Indeed, if a scheme is locally Noetherian,
then all classes of locally absolutely pure, absolutely pure,
locally injective and injective quasi-coherent sheaves are equal, by
\cite[II, proposition 7.17, theorem 7.18]{Hartshorne}.

For the ``only if'' part, suppose that the class of locally
injective and locally absolutely pure quasi-coherent sheaves are
equal. As explained above, we deal with  a cover $\{D_+( \prod_{i
\in v}x_i)\}_{v\subseteq \{1,\ldots,n\}}$ of basic affine open
subsets of $X$ since locally absolutely purity  is
independent of choice of the base by proposition \ref{ap3}.  Let $M$
be an absolutely pure $R[v]$-module for some $v\subseteq
\{1,\ldots,n\}$. By taking its direct image $\iota_
* (\widetilde{M})$, we get a locally absolutely pure quasi-coherent
sheaf on $X$. Indeed,  $\iota_ *
(\widetilde{M})(D_+(\prod_{i \in w}x_i))=S_{vw}^{-1}M(v)$ for $v
\subseteq w $ is absolutely pure $R[w]$-module by
\cite[theorem 3.21]{Pinzon} and $\iota_ *
(\widetilde{M})(D_+(\prod_{i \in w}x_i))= \widetilde{M}(D_+(\prod_{i
\in w}x_i) \cap D_+(\prod_{i \in v}x_i))$ as $R[w]$-module for $v
\nsubseteq w$. But $$\widetilde{M}(D_+(\prod_{i \in w}x_i) \cap
D_+(\prod_{i \in v}x_i))= S_{v(v \cup w)}^{-1}M(v)$$ is absolutely
pure as $R[(v \cup w)]$-module and since $R[(v \cup w)]=S_{v(v \cup
w)}^{-1} R[w]$, it is also absolutely pure as $R[w]$-module, by
\cite[theorem 3.20]{Pinzon}. By assumption $\iota_ *
(\widetilde{M})$ is locally injective, that is, $(\iota_ *
(\widetilde{M})) (D_+(\prod_{i \in v}x_i))=M$ is injective. So,
$R[v]$ is Noetherian, by \cite[theorem 3]{Megibben}. This implies
that $X$ is locally Noetherian.
\end{proof}

Note that the class of locally absolutely pure quasi-coherent
sheaves over a locally coherent scheme is closed under direct limits
and coproducts since absolutely pure modules over  coherent rings
are closed under direct limits, \cite[proposition 2.4]{Pinzon}.
\begin{theorem}
Let $X$ be a locally coherent scheme. The class of locally absolutely pure quasi-coherent sheaves is a covering class.
\end{theorem}
\begin{proof}
First note that over a  coherent ring, a quotient of an absolutely pure
module by a pure submodule is again absolutely pure \cite[proposition 4.2]{Pinzon}. So, using that,
we can say that a quotient of a locally absolutely pure by a pure quasi-coherent subsheaf is again locally absolutely pure.

Let $\lambda$ be the cardinality of the scheme $X$, that is, the
supremum of all cardinalities of $\O_X(U)$ for all  affine open
subset $U \subseteq X$. By \cite[corollary 3.5]{EE}, there is an
infinite cardinal $\kappa$ such that every quasi-coherent sheaf
can be written as a sum of quasi-coherent subsheaves of type
$\kappa$. In fact, every subsheaf with type $\kappa$ of a
quasi-coherent sheaf $\scrF$ can be embedded in a quasi-coherent
subsheaf of type $\kappa$ which is pure in $\scrF
$. Let $\mathcal{S}$
be the set of locally absolutely pure quasi-coherent sheaves of type
$\kappa$. By combining this  with the fact that the class of locally
absolutely pure quasi-coherent sheaves is closed under taking a
quotient by a pure quasi-coherent sheaf, it follows that each locally
absolutely pure quasi-coherent sheaf admits an
$\mathcal{S}$-filtration. So, every locally absolutely pure sheaf is
filtered  by those of type $\kappa$.

On the other hand, since absolutely pure modules are closed under
extensions and direct limits over a coherent ring, every
quasi-coherent sheaf on a locally coherent scheme possessing an
$\mathcal{S}$-filtration is also locally absolutely pure
quasi-coherent. So, the class of locally absolutely pure
quasi-coherent sheaves is  equal to the class
$\Filt(\mathcal{S})$ of all $\mathcal{S}$-filtered quasi-coherent sheaves. So, that class is precovering. Being closed
under direct limits also implies that the class of locally absolutely
pure quasi-coherent sheaves is covering.
\end{proof}

\end{document}